\documentclass{amsart}

\usepackage{amsmath, amssymb, amsthm, amsfonts, amsbsy, amscd}
\usepackage{mathrsfs, latexsym, eucal}

\usepackage{graphicx}
\usepackage{color}
\usepackage{verbatim}
\usepackage[hidelinks]{hyperref}

\newtheorem{corollary}{Corollary}

\newtheorem{lemma}{Lemma}
\newtheorem{proposition}{Proposition}
\newtheorem{remark}{Remark}
\newtheorem{theorem}{Theorem}

\numberwithin{equation}{section}

\newcommand{\vol}{\mathrm{Vol}}
\newcommand{\ric}{\mathrm{Ric}}

\newcommand{\lr}[2]{\langle{{#1},{#2}}\rangle}

\newcommand{\rict}{\mathring{\operatorname{Ric}}}
\newcommand{\diam}{\mathrm{diam}(M,g)}
\newcommand{\lrsf}{\lr{\nabla R}{\nabla f}}

\begin{document}
	\title[Four-dimensional compact quasi-Einstein manifolds]{Diameter estimates and Hitchin--Thorpe inequality for four-dimensional compact quasi-Einstein manifolds }
	\author{Samuel Belo}

	\address{Departamento de Matem\'atica, Universidade Federal do Cear\'a - UFC, Campus do Pici, Av. Humberto Monte, Bloco 914,
		60455-760, Fortaleza - CE, Brazil}\email{sanetto@alu.ufc.br}

	\thanks{S. Belo was partially supported by CAPES/Brazil - Finance Code 001}


	
	\begin{abstract}	
		We study compact $m$-quasi-Einstein manifolds and derive geometric estimates relating the oscillation of the potential function to the diameter of the manifold. We obtain lower bounds for the diameter in terms of the oscillation of the potential function. As an application in dimension four, we derive diameter conditions ensuring that compact $m$-quasi-Einstein manifolds satisfy the Hitchin--Thorpe inequality. Our results extend diameter estimates in smooth metric measure spaces and are consistent with known bounds in the limiting case corresponding to Ricci solitons. Finally, we provide a volume estimate involving the oscillation.
	\end{abstract}
	\maketitle
	\section{Introduction}
	
	Let $(M^n,g)$ be a complete Riemannian manifold and let $f\in C^\infty(M)$ be a smooth function. 
	The $m$-Bakry–Émery Ricci tensor is defined by
	\begin{align}\label{eq:ricci_bakry_emery}
		\operatorname{Ric}_f^m = \operatorname{Ric} + \nabla^2 f - \frac{1}{m} df\otimes df,
	\end{align}
	where $m>0$ is a constant.
	
	When this tensor is a constant multiple of the metric, the manifold is called 
	an $m$-quasi-Einstein manifold. In other words, $(M^n,g,f,m)$ satisfies
	\begin{align}\label{eq:quasi_einstein}
		\operatorname{Ric}_f^m = \lambda g,
	\end{align}
	for some constant $\lambda$.
	
	The $m$-quasi-Einstein metrics arise naturally in the study of smooth metric 
	measure spaces and Einstein warped products. In particular, when $m$ is a 
	positive integer, an $m$-quasi-Einstein manifold corresponds to the base of an 
	Einstein warped product whose fiber has dimension $m$, see \cite{Besse, catino2013, He2012,  kim2003, mastrolia2014, rimoldi2011, wang2014, catino2012}. 
	Although this geometric condition can be formulated for a general vector field $X$ 
	by replacing the Hessian with the Lie derivative $\frac{1}{2}\mathcal{L}_X g$, see \cite{Barros2012, Limoncu2010}, 
	we restrict our attention to the gradient case, which is the primary focus in the compact setting. 
	From the analytic point of view, these metrics can also be interpreted as a finite-dimensional 
	analogue of gradient Ricci solitons, which correspond formally to the limiting case $m=\infty$, see \cite{cao2010, Case2011, ChengRibeiroZhou2023, hamilton1995}. Although $m$-quasi-Einstein manifolds and gradient Ricci solitons share formal similarities, they exhibit fundamental geometric differences, as illustrated by specific examples; see, e.g., \cite{He2012, barros2014, bohm1998, Case2011}.
	
	A complete $m$-quasi-Einstein manifold is compact if and only if $\lambda > 0$ and $m < \infty$, see \cite{bakry1985, qian1997}. Therefore, in this case, the study of quasi-Einstein metrics naturally reduces to the compact setting. From now on, all compact manifolds considered in this paper are assumed to be connected and without boundary (closed). Rigidity phenomena are prevalent in this context and severely restrict the existence of non-trivial solutions. For instance, if $\lambda \le 0$, it is a well-known fact that any compact $m$-quasi-Einstein manifold is trivial, meaning its potential function is constant and the underlying metric is Einstein. A similar triviality occurs if the scalar curvature is constant. For foundational rigidity results and recent classifications under these conditions, we refer to \cite{Case2011, Costa2025, diogenes2022compact, diogenes2022remarks, he2014warped}. Beyond the closed setting, an extensive literature explores non-compact quasi-Einstein manifolds and generalizations involving manifolds with boundary. For related rigidity and classification theorems in these contexts, we refer to \cite{catino2013, Cheng2020, Diogenes2025, He2012, ribeiro2021, wang2011, bohm1999} and the references therein.
	
	In dimension two, every compact solution is trivial, since the quasi-Einstein equation reduces essentially to the Einstein condition. A similar rigidity phenomenon occurs in dimension three, where every compact quasi-Einstein manifold is 
	necessarily Einstein, see \cite{Case2011}. In dimension four, non-trivial compact quasi-Einstein metrics do exist. Lü, Page, and Pope~\cite{LuPagePope2004} constructed a continuous family of compact quasi-Einstein metrics, depending on the parameter $m$, on the manifold $\mathbb{CP}^2 \# \overline{\mathbb{CP}}^{\,2}$. These metrics arise from a cohomogeneity-one construction closely related to the geometry of Einstein metrics and Ricci solitons on similar manifolds. They are not Kähler, although they remain closely connected to Kähler geometry through a conformal structure. In the limit $m\to\infty$, this family converges to the well-known Koiso--Cao gradient shrinking Ricci soliton \cite{koiso1993}.
	
	Another notable example of a compact four-dimensional Ricci soliton is the one constructed by Wang and Zhu on $\mathbb{CP}^2 \# 2\overline{\mathbb{CP}}^{\,2}$~\cite{WangZhu2004}. However, the existence of a corresponding quasi-Einstein metric on this space remains an open problem (see, for instance, \cite{Batat2015}). At present, the Lü--Page--Pope construction provides the only explicit family of non-trivial compact quasi-Einstein metrics in dimension four, illustrating the restrictive nature of the quasi-Einstein condition. To the best of our knowledge, no general classification or further explicit constructions are currently available.
	
	Even in the compact case and in dimension four, the interaction between the geometry of curvature and the global topology of the manifold is particularly rich. A fundamental manifestation of this relationship is the Hitchin--Thorpe inequality, which establishes a remarkable relation between the Euler characteristic $\chi(M)$ and the signature $\tau(M)$ of a manifold:
	\begin{align}\label{eq:hitchin_thorpe}
		2\chi(M) \pm 3\tau(M) \ge 0 .
	\end{align}
	The rigid geometric condition of being Einstein is sufficient to impose this topological restriction, providing an obstruction to the existence of Einstein metrics, see \cite{thorpe1969,Hitchin1974}. However, the converse is not generally true, which motivates the search for other geometric conditions that also yield the Hitchin--Thorpe inequality.
	
	In this context, H.-D. Cao raised the following question for Ricci solitons: \textit{“Is the Hitchin--Thorpe inequality valid for compact four-dimensional gradient shrinking Ricci solitons?”}, see \cite{cao2010}. In the recent years, some partial answers have been obtained, see, for instance, \cite{lima2013, FernandezLopez2010, tadano2018, zhang2012, ChengRibeiroZhou2023, brasil2014}.
	
	Since quasi-Einstein manifolds arise as a natural generalization of Einstein manifolds and are closely related to Ricci solitons, it is natural to investigate whether the Hitchin--Thorpe inequality also holds in this setting. In particular, one may ask under which geometric or analytical conditions a quasi-Einstein manifold must satisfy the Hitchin--Thorpe inequality.
	
	We establish several sufficient conditions under which compact quasi-Einstein manifolds satisfy the Hitchin--Thorpe inequality. The first result reads as follows. 
	
	\begin{proposition}\label{proposition1.1}
		Let $(M^4,g,f,m)$ be a compact quasi-Einstein manifold of dimension 4 with $m>1$. If the scalar curvature $R$ satisfies the following integral inequality:
		\begin{align}\label{Theo1Form1}
			\int_M R^2 dV_g \leq \dfrac{24(m+1)}{m + 2}\lambda^2 \vol(M),
		\end{align}
		then $M^4$ satisfies the Hitchin--Thorpe inequality.
	\end{proposition}
	
	The next result is inspired by a recent partial answer to Cao’s Problem for compact Ricci solitons obtained by Cheng, Ribeiro and Zhou \cite{ChengRibeiroZhou2023}.
	
	\begin{theorem}\label{theorem2}
		Let $(M^4,g,f,m)$ be a compact quasi-Einstein manifold of dimension 4 with $m>1$. Then, the following estimate holds:
		\begin{align}\label{Desig.Teo.2}
			8 \pi^2\chi(M) &\geq \int_M |W|^2 dV_g + \dfrac{m^2\lambda^2 }{6(m - 1)(m + 3)}\vol(M) \left(5 + \dfrac{8}{m} - \dfrac{12}{m^2} - e^{\frac{f_{osc}(m + 2)}{m}}\right)
		\end{align}
		where $f_{osc} = f_{\max} - f_{\min}$, with $f_{\max}$ and $f_{\min}$ being the maximum and minimum of the potential function on $M^4$. Furthermore, equality holds if and only if $f$ is constant.
	\end{theorem}
	
	As a consequence of this theorem, we obtain an estimate ensuring the
	validity of the Hitchin--Thorpe inequality for quasi-Einstein manifolds.
	
	\begin{corollary}\label{CorTheorem2}
		Let $(M^4,g,f,m)$ be a compact quasi-Einstein manifold of dimension 4 with $m>1$. If
		\begin{align*}
			f_{\max} - f_{\min} &\leq \dfrac{m}{m+2}\log\left(5 + \dfrac{8}{m} - \dfrac{12}{m^2}\right)	
		\end{align*}
		then $M^4$ satisfies the Hitchin--Thorpe inequality.
	\end{corollary}
	
	\begin{remark}
		The compact quasi-Einstein metrics constructed by Lü, Page and Pope~\cite{LuPagePope2004} satisfy the oscillation bound obtained in Corollary~\ref{CorTheorem2}. In particular, each element of their family provides an explicit example realizing the estimate for the oscillation of the potential function. An explicit description of the potential function and the corresponding values of its extrema can be obtained using the method developed in \cite{Batat2015}.
	\end{remark}
	
	\begin{remark}
		In the formal limit $m\to\infty$, the quasi-Einstein equation reduces to the gradient Ricci soliton equation. In this case, Proposition~\ref{proposition1.1} recovers the result of Ma~\cite{lima2013} for compact gradient shrinking Ricci solitons. Similarly, the estimate in Corollary~\ref{CorTheorem2} reduces to
		\begin{align*}
			f_{\max} - f_{\min} &\leq \log\left(5\right)
		\end{align*}
		This is precisely the estimate obtained by Cheng, Ribeiro and Zhou in~\cite{ChengRibeiroZhou2023}.
	\end{remark}

	Diameter estimates for quasi-Einstein manifolds and, more generally, for smooth metric measure spaces have been investigated by several authors; see for instance \cite{WeiWylie2007,Shaikh2024, Wu2018, Wang2013, Deng2021, Deng2024, tadano2026compactness, tadano2020compactness, tadano2022riccati, tadano2025myers, tadano2025improvement}. These results typically relate the diameter of the manifold to curvature bounds and properties of the potential function. Motivated by these works, we obtain estimates that connect the diameter of the manifold with the oscillation of the potential function.
	
	In the following result we derive lower bounds for the diameter of the manifold in terms of the oscillation of the potential function and bounds on the Ricci curvature.
	
	Since $M$ is compact, the function $(x,v)\mapsto \operatorname{Ric}(v,v)$ on the unit tangent bundle attains its minimum and maximum. We denote
	\begin{align*}
		c := \min_{|v|=1}\operatorname{Ric}(v,v), \qquad C := \max_{|v|=1}\operatorname{Ric}(v,v).
	\end{align*}
	These constants will be used in the estimates below. Let us also denote $d = \diam$.
	
	\begin{theorem} \label{Theorem3}
		Let $(M^n,g,f,m)$ be a compact non-trivial $m$-quasi-Einstein manifold of dimension $n$ with $m>0$. If $f_{osc} = f_{\max} - f_{\min}$ denotes the oscillation of the potential function $f$, then the following diameter estimates hold:
		\begin{enumerate}
			\item $d \ge \sqrt{\frac{m}{\lambda - c}} \arccos\left( e^{-\frac{f_{osc}}{m}} \right)$,
			\item $d \ge \sqrt{\frac{m}{C - \lambda}} \operatorname{arccosh}\left( e^{\frac{f_{osc}}{m}} \right).$
		\end{enumerate}
	\end{theorem}
	
	We also obtain the following mixed estimate for the oscillation of the
	potential function.
	
	\begin{proposition}\label{Prop:MixedEstimate}
		Let $(M^n,g,f,m)$ be a compact non-trivial $m$-quasi-Einstein manifold.
		If the diameter satisfies $d < {\pi}\sqrt{\frac{m}{\lambda-c}}$,
		then 
		\begin{align*}
			e^{f_{osc}/m} \le \cosh\!\left(\sqrt{\frac{C-\lambda}{m}}\frac d2\right) \sec\!\left(\sqrt{\frac{\lambda-c}{m}}\frac d2\right).
		\end{align*}
	\end{proposition}
	
	\begin{remark}
		It is worth noting that the diameter bounds obtained in Theorem~\ref{Theorem3}
		as well as the mixed estimate in Proposition~\ref{Prop:MixedEstimate} are consistent with
		the Ricci soliton case. Indeed, in the formal limit $m\to\infty$,
		the three bounds converge to the diameter estimates appearing
		in \cite[Theorem 1.4]{FernandezLopez2010} for compact four-dimensional Ricci solitons.
	\end{remark}
	
	As a direct consequence of Theorem~\ref{theorem2} and Theorem~\ref{Theorem3}, and using the diameter estimates obtained above, we derive the following conditions ensuring that the Hitchin--Thorpe inequality holds. Let us consider
	\begin{align*}
		D_m = \sqrt[m+2]{5 + \frac{8}{m} - \frac{12}{m^2}}
	\end{align*}
	\begin{corollary} \label{CorTheorem3}
		Let $(M^4,g,f,m)$ be a compact non-trivial $4$-dimensional $m$-quasi-Einstein manifold with $m>1$. If the diameter of $M$ satisfies one of the following conditions
		\begin{enumerate}
			\item $d \le \sqrt{\frac{m}{\lambda - c}} \arccos\left( \frac{1}{D_m} \right);$
			\item $d \le \sqrt{\frac{m}{C - \lambda}} \operatorname{arccosh}\left( D_m \right);$
			\item $d \le 2x_0$, where $x_0$ is the unique real solution in the interval $\left(0, \frac{\pi}{2}\sqrt{\frac{m}{\lambda - c}}\right)$ to the equation:
			$$ \cosh\left( \sqrt{\frac{C - \lambda}{m}} x_0 \right) \sec\left( \sqrt{\frac{\lambda - c}{m}} x_0 \right) = D_m $$
		\end{enumerate}
		then the Hitchin--Thorpe inequality is satisfied. 
	\end{corollary}
	
	\begin{remark}
		In the formal limit $m\to\infty$, corresponding to the gradient Ricci soliton equation, the estimates obtained in Corollary~\ref{CorTheorem3} recover the diameter bound for compact Ricci solitons obtained in \cite[Theorem~1.1]{FernandezLopez2010}, while Corollary~\ref{CorTheorem2} reduces to the Hitchin--Thorpe criterion established in \cite[Corollary~1]{ChengRibeiroZhou2023}. Combining these two results, we obtain the following diameter condition: if a compact four-dimensional gradient Ricci soliton $(M^4,f,g)$ satisfies
		\begin{align}\label{diameter.log}
			d(M,g) \le \max\left\{
			\sqrt{\dfrac{2\log(5)}{\lambda - c}},
			\sqrt{\dfrac{2\log(5)}{C - \lambda}},
			2\sqrt{\dfrac{2\log(5)}{C - c}}
			\right\},
		\end{align}
		then $M^4$ must satisfy the Hitchin--Thorpe inequality. 
		
		Moreover, whenever $\log(5)>1$, the bound \eqref{diameter.log} improves upon the estimate established by Fernández-López and García-Río~\cite[Theorem 1.4]{FernandezLopez2010} by providing a larger upper bound for the diameter. In particular, this shows that the quasi-Einstein estimates obtained here are consistent with the corresponding bounds known in the Ricci soliton setting.
	\end{remark}
	Finally, as a consequence of Theorem~\ref{theorem2}, we prove an estimate for the volume of a compact quasi-Einstein manifold.	
	\begin{theorem} \label{Theorem.Est.Vol}
		Let $(M^4, g, f, m)$ be a compact $m$-quasi-Einstein manifold of dimension $4$ with $m > 1$. Then the following estimate for the volume of $M$ holds:
		\begin{align}\label{Eq.1Teo4}
			\frac{96\pi^2}{\lambda^2} \geq \frac{m^2}{(m - 1)(m + 3)} \left( 5 + \frac{8}{m} - \frac{12}{m^2} - e^{\frac{f_{osc}(m + 2)}{m}} \right) \vol(M).
		\end{align}
		Equality is satisfied if and only if $M$ is isometric to a sphere of radius $\sqrt{3/\lambda}$. Furthermore, the Yamabe invariant satisfies:
		\begin{align}\label{Eq.2Teo4}
			\mathcal{Y}(M, [g])^2 \geq \frac{4m^2\lambda^2}{(m - 1)(m + 3)} \left( 5 + \frac{8}{m} - \frac{12}{m^2} - e^{\frac{f_{osc}(m + 2)}{m}} \right) \vol(M),
		\end{align}
		and this equality holds if and only if $M$ is Einstein.
	\end{theorem}
	
	
	\section{Preliminaries}

	In this section, we introduce basic notations as well as some  preliminary results that will be used throughout the paper. First, we recall some tools derived from the Riemann curvature tensor. The first is the Weyl tensor $ W_{ijkl}$ which splits as 
	\begin{align}\label{TensorRm}
		R_{ijkl} &= W_{ijkl} + \dfrac{1}{n - 2}(R_{ik}g_{jl} + R_{jl}g_{ik} - R_{il}g_{jk} - R_{jk}g_{il})\nonumber\\
		&\quad - \dfrac{R}{(n - 1)(n-2)}(g_{jl}g_{ik} - g_{il}g_{jk}),
	\end{align}
	where $g_{ij}, R_{ij}=g^{kl}R_{iklj}$  and $R_{ijkl}$	stand, respectively,  for the components of the metric, the Ricci tensor and the  curvature tensor. The scalar curvature is given by $R= g^{ij}R_{ij}.$
	
	In 4-dimensional Riemannian geometry, the analysis of curvature often focuses on the algebraic structure of the Weyl tensor, $W$. To simplify this analysis, one exploits a remarkable property of the space of 2-forms ($\Lambda^2$) in this dimension. This space naturally decomposes into two three-dimensional subspaces, known as the space of self-dual ($\Lambda^+$) and anti-self-dual ($\Lambda^-$) 2-forms. This division is canonical, conformally invariant, and expressed as a direct sum:
	\begin{equation}
		\Lambda^2 = \Lambda^+ \oplus \Lambda^-.
	\end{equation}
	The Weyl tensor can be viewed as an operator acting on the space $\Lambda^2$. The key advantage of this decomposition is that the Weyl tensor preserves it; that is, it never maps an element from $\Lambda^+$ to $\Lambda^-$ or vice versa. This allows its action to be studied independently on each subspace, splitting the tensor into two simpler parts:
	\begin{equation}
		W = W^+ \oplus W^-,
	\end{equation}
	where $W^\pm : \Lambda^\pm M \longrightarrow \Lambda^\pm M$ are endomorphisms known, respectively, as the \textit{self-dual part} and the \textit{anti-self-dual part} of $W$.
	
	From this decomposition, the curvature operator is given by the following matrix
	\begin{equation}\label{matriz Rm}
		\mathcal{R} = \left(\begin{array}{@{}cc@{}}
			\begin{matrix}
				W^+ + \frac{R}{12}g\\
			\end{matrix}
			& \mathring{\text{Ric}} \\
			\mathring{\text{Ric}} &
			\begin{matrix}
				W^- + \frac{R}{12}g\\
			\end{matrix}
		\end{array}\right),
	\end{equation}
	where $R$ is the scalar curvature and $\mathring{\text{Ric}}$ is the traceless part of the Ricci tensor.
	At any given point $p \in M^4$, by choosing an oriented orthonormal basis $\{e_i\}_{i=1}^4$, we can construct an explicit basis for the $\Lambda^\pm$ subspaces:
	\begin{equation}
		\{ e^1 \wedge e^2 \pm e^3 \wedge e^4, e^1 \wedge e^3 \pm e^4 \wedge e^2, e^1 \wedge e^4 \pm e^2 \wedge e^3 \},
	\end{equation}
	where each of these bivectors has  norm  $\sqrt{2}$.
	
	Finally, since $W^+$ and $W^-$ are operators on 3-dimensional spaces, they each have three eigenvalues, denoted by $w_i^\pm$ (for $i=1,2,3$). The traceless property of the Weyl tensor implies that the sum of these eigenvalues must also be zero for each part, leading to the following conditions:
	\begin{equation}
		w_1^\pm \le w_2^\pm \le w_3^\pm \quad \text{and} \quad w_1^\pm + w_2^\pm + w_3^\pm = 0.
	\end{equation}
	On the other hand, for a compact Riemannian manifold $(M^{4},g)$, there are two classical formulas that relate the geometry and the topology of the manifold. The first one is the Gauss--Bonnet--Chern formula given by
	\begin{align}\label{Gauss--Bonnet--Chernformula}
		8\pi^2 \chi(M) = \int_M \left(|W^+|^2 + |W^-|^2 + \dfrac{R^2}{24} - \dfrac{1}{2}|\mathring{\text{Ric}}|^2\right)dV_g,
	\end{align}
	whereas the second one is the Hirzebruch formula
	\begin{align}\label{Hirzebruchformula}
		12 \pi^2 \tau(M) = \int_M (|W^+|^2 - |W^-|^2) dV_g.
	\end{align}		
	For the foundational theory regarding the Gauss--Bonnet--Chern formula and the Hirzebruch signature theorem, we refer the reader to \cite[Chapter 13]{Besse}. As a consequence of these two distinguished formulas  we arrive at the following identity
	\begin{align}\label{gauss-hei}
		2\chi(M) \pm 3\tau(M) = \dfrac{1}{4 \pi^2}\int_M \left(2| W^{\pm}|^2 + \dfrac{R^2}{24} - \dfrac{1}{2}|\mathring{\text{Ric}}|^2\right)dV_g
	\end{align}	
	It is easy to see that the Hitchin--Thorpe inequality~\eqref{eq:hitchin_thorpe} is satisfied whenever the traceless part of the Ricci tensor vanishes.
	
	In the following, we recall the useful equations for the curvatures of an $m$-quasi-Einstein manifold. For their proofs and a more detailed discussion, we refer the reader to \cite{Case2011} and the references therein.
	\begin{lemma}\label{Lemma1}
		Let $(M^4,g,f,m)$ be a quasi-Einstein manifold of dimension $4$ with $m > 1$, then the following relations are satisfied:
		\begin{itemize}
			\item [1)] $R + \Delta f - \dfrac{1}{m}|\nabla f|^2 = 4 \lambda $
			\item [2)] $\dfrac{1}{2} \nabla R = \dfrac{m - 1}{m} \ric (\nabla f) + \dfrac{1}{m}(R - 3\lambda)\nabla f$
			\item [3)] $\Delta R = \dfrac{m + 2}{m}\lrsf - \dfrac{2(m-1)}{m}|\ric|^2 - \dfrac{2}{m}R^2 + \dfrac{2(m+6)}{m}\lambda R - \dfrac{24}{m}\lambda^2$ 
		\end{itemize}
	\end{lemma}
	
	In what follows, we establish a key lemma that will be essential for the proofs of our main results, specifically Proposition~\ref{proposition1.1} and Theorem~\ref{theorem2}.
	
	\begin{lemma}\label{Lemma2}
		Let $(M^4,g,f,m)$ be a quasi-Einstein manifold with $m>1$, then
		\begin{align}\label{Euler.Scal}
			8\pi^2 \chi(M) &= \int_M |W|^2 dV_g + \dfrac{(m- 2)\lambda}{2m(m - 1)} \int_M|\nabla f|^2 dV_g + \dfrac{m + 2}{4m(m - 1)}\int_M R|\nabla f|^2 dV_g \\
			&\quad - \dfrac{m + 2}{12(m - 1)}\int_M R^2 dV_g + 2\dfrac{m+1}{m-1}\lambda^2 \vol(M);\nonumber
		\end{align}
		\begin{align}\label{Euler.Grad}
			8\pi^2 \chi(M) &= \int_M |W|^2 dV_g + \dfrac{(m- 10)\lambda}{6m(m - 1)} \int_M|\nabla f|^2 dV_g + \dfrac{m + 2}{6m(m - 1)}\int_M R|\nabla f|^2 dV_g\\
			&\quad - \dfrac{m + 2}{12(m - 1)}\int_M \lrsf dV_g + \dfrac{2}{3}\lambda^2 \vol(M).\nonumber
		\end{align}
	\end{lemma}
	\begin{proof}
		From the first item of Lemma~\ref{Lemma1}, we have
		\begin{align}\label{IntR}
			\int_M R dV_g =  4\lambda \vol(M) + \dfrac{1}{m}\int_M |\nabla f|^2 dV_g.
		\end{align}
		Now, note that
		\begin{align}\label{IntR^2}
			\int_M R^2 dV_g &= \int_M R(4 \lambda + \dfrac{1}{m}|\nabla f|^2 -\Delta f)dV_g\nonumber\\
			&=  4 \lambda \int_M R dV_g + \dfrac{1}{m} \int_M R|\nabla f|^2 dV_g - \int_M R \Delta f dV_g\nonumber\\
			&= 16\lambda^2\vol(M) + \dfrac{4\lambda}{m}\int_M |\nabla f|^2 dV_g + \dfrac{1}{m}\int_M R|\nabla f|^2 dV_g + \int_M \lrsf dV_g.
		\end{align}
		Integrating the third item of Lemma~\ref{Lemma1}, it follows that			
		\begin{align*}
			\int_M |\ric|^2 &= \dfrac{m + 2}{2(m - 1)}\int_M \lrsf dV_g - \dfrac{1}{m - 1}\int_M R^2 dV_g + \dfrac{(m + 6)\lambda}{m - 1} \int_M R dV_g\\
			& \quad - \dfrac{12\lambda^2}{m - 1} \vol(M)\\
			&= \dfrac{m + 2}{2(m - 1)}\int_M \lrsf dV_g - \dfrac{1}{m - 1}\int_M R^2 dV_g + \dfrac{4(m + 3)}{m - 1}\lambda ^2 \vol(M)\\
			& \quad+ \dfrac{(m + 6)\lambda}{m(m - 1)} \int_M |\nabla f|^2 dV_g\\
			&=\dfrac{m}{2(m - 1)}\int_M \lrsf dV_g + \dfrac{(m + 2)\lambda}{m(m - 1)} \int_M |\nabla f|^2 dV_g + 4\lambda^2 \vol(M)\\
			&\quad - \dfrac{1}{m(m - 1)}\int_M R|\nabla f|^2 dV_g,
		\end{align*}
		where we used \eqref{IntR} and \eqref{IntR^2}. 
		Now, substituting \eqref{IntR^2} into the relation above, we obtain
		\begin{align}\label{RicScal}
			\int_M |\ric|^2 &= \dfrac{m}{2(m-1)}\int_M R^2 dV_g - \dfrac{(m - 2)\lambda}{m(m - 1)} \int_M |\nabla f|^2 dV_g\\
			&\quad - \dfrac{m+2}{2m(m - 1)}\int_M R|\nabla f|^2 dV_g - 4\dfrac{m+1}{m-1}\lambda^2 \vol(M)\nonumber
		\end{align}	
		By substituting \eqref{RicScal} into \eqref{Gauss--Bonnet--Chernformula}, it follows that
		\begin{align*}
			8\pi^2 \chi(M) &= \int_M |W|^2 dV_g + \dfrac{(m- 2)\lambda}{2m(m - 1)} \int_M|\nabla f|^2 dV_g + \dfrac{m + 2}{4m(m - 1)}\int_M R|\nabla f|^2 dV_g \nonumber\\
			&\quad - \dfrac{m + 2}{12(m - 1)}\int_M R^2 dV_g + 2\dfrac{m+1}{m-1}\lambda^2 \vol(M).
		\end{align*}
		This proves \eqref{Euler.Scal}. Substituting \eqref{IntR^2} into the formula above, we obtain
		\begin{align*}
			8\pi^2 \chi(M) &= \int_M |W|^2 dV_g + \dfrac{(m- 10)\lambda}{6m(m - 1)} \int_M|\nabla f|^2 dV_g + \dfrac{m + 2}{6m(m - 1)}\int_M R|\nabla f|^2 dV_g \nonumber\\
			&\quad - \dfrac{m + 2}{12(m - 1)}\int_M \lrsf dV_g + \dfrac{2}{3}\lambda^2 \vol(M),
		\end{align*}
		which proves \eqref{Euler.Grad} and thereby completes the proof of the lemma.
	\end{proof}
	
	\section{Proof of Main Results}
	
	\begin{proof}[Proof of Proposition~\ref{proposition1.1}]
		Using \eqref{Euler.Scal} and the Hirzebruch formula~\eqref{Hirzebruchformula},	we obtain
		\begin{align*}
			4 \pi^2\left(2  \chi(M) \pm 3\tau(M)\right) &=  \int_M 2|W^{\pm}|^2 dV_g - \dfrac{m + 2}{12(m-1)} \int_M R^2 dV_g + \dfrac{(m - 2)\lambda}{2m(m-1)}\int_M |\nabla f|^2 dV_g\\
			&\quad + \dfrac{m + 2}{4m(m-1)}\int_M R|\nabla f|^2 dV_g + \dfrac{2(m+1)}{m - 1}\lambda^2 \vol(M).
		\end{align*}
		Using the lower bound $R \geq \frac{12}{m+3}\lambda$ established in \cite{Case2011}, we can estimate the gradient terms in the expression above. A direct computation yields:
		\begin{align*}
			\dfrac{(m - 2)\lambda}{2m(m-1)} + \dfrac{m + 2}{4m(m-1)}R &\geq \dfrac{\lambda}{4m(m-1)}\left(2(m-2) + \dfrac{12(m+2)}{m+3}\right)\\
			&= \dfrac{2m^2 + 14m + 12}{4m(m-1)(m+3)}\lambda.
		\end{align*}
		Since $m > 1$, this coefficient is strictly positive. Therefore, the integral involving $|\nabla f|^2$ is non-negative and can be bounded from below by zero. Thus, if
		\begin{align*}
			\dfrac{m + 2}{12(m-1)} \int_M R^2 dV_g \leq \dfrac{2(m+1)}{m - 1}\lambda^2 \vol(M),
		\end{align*}
		then $2  \chi(M) \pm 3\tau(M) \geq 0$, in other words, if
		\begin{align*}
			\int_M R^2 dV_g \leq \dfrac{24(m+1)}{m + 2}\lambda^2 \vol(M),
		\end{align*}
		the Hitchin--Thorpe inequality is satisfied. This proves \eqref{Theo1Form1}.
	\end{proof}

	For the proof of Theorem~\ref{theorem2}, we will use the following result due to Colding and Minicozzi, see Lemma 1.3 in \cite{colding2021}.
	
	\begin{lemma}\label{TechLemma1}
		Suppose that $\alpha$ is a proper $C^n$ function and $\mathcal{H}^{n}\{| \nabla \alpha|  = 0\} = 0$ in $\{\alpha \geq r_0\}$ for some fixed $r_0$. If $h$ is a bounded function and $Q(r) = \int_{r_0 < \alpha < r} h$, then $Q$ is absolutely continuous and $Q'(r) = \int_{\alpha = r} \frac{h}{| \nabla \alpha| }$ almost everywhere.
	\end{lemma}
	
	\begin{proof}[Proof of Theorem~\ref{theorem2}]
		Consider the sub-level sets $D(s) = \{x : f(x) < s\}$ and let $a=f_{\min}$, $b=f_{\max}$. According to \cite{He2012}, $g$ and $f$ are real analytic. In particular, $\mathcal{H}^4\{|\nabla f|  = 0\} = 0$. Therefore, applying Lemma~\ref{TechLemma1} and the divergence theorem, we obtain for any $s \in [a,b]$:
		\begin{align*}
			\int_{D(s)} \langle \nabla R, \nabla f \rangle dV_g &= \int_{a}^{s} \int_{\partial D(t)} \dfrac{\langle \nabla R, \nabla f \rangle}{|\nabla f|} d \sigma dt = \int_{a}^{s} \left( \int_{D(t)} \Delta R dV_g\right)dt\\
			&= \int_{a}^{s} \left( \int_{D(t)} \left[ \dfrac{m + 2}{m}\langle \nabla R, \nabla f \rangle - \dfrac{2(m-1)}{m}|\operatorname{Ric}|^2 - \dfrac{2}{m}R^2 \right.\right.\\
			&\quad + \left.\left.\dfrac{2(m+6)}{m}\lambda R - \dfrac{24}{m}\lambda^2 \right] dV_g\right)dt.
		\end{align*}
		Define the following functions:
		\begin{align*}
			\phi(s) &= \int_{a}^{s}\left(\int_{D(t)}\langle \nabla R, \nabla f \rangle dV_g\right)dt, \\
			\varphi(s) &= \int_{a}^{s}\left(\int_{D(t)} \left[- \dfrac{2(m-1)}{m}|\operatorname{Ric}|^2 - \dfrac{2}{m}R^2 + \dfrac{2(m+6)}{m}\lambda R - \dfrac{24}{m}\lambda^2\right] dV_g\right)dt.
		\end{align*}
		It follows from the fundamental theorem of calculus and the integral relations above that $\phi''(s) = \dfrac{m+2}{m} \phi'(s) + \varphi'(s)$. Given that $\phi'(a) = 0$, setting $K = \dfrac{m + 2}{m}$ for convenience, we solve the linear differential equation to obtain:
		\begin{align*}
			\phi'(s) = e^{K s} \int_{a}^{s} \varphi'(t)e^{-Kt}dt.
		\end{align*}
		Evaluating at $s=b$, we bound $\phi'(b)$ by completing the square for $R$:
		\begin{align}\label{Estimative.Oscf}
			\phi'(b) &= e^{K b} \int_{a}^{b} \left(\int_{D(t)} \left[- \dfrac{2(m-1)}{m}|\operatorname{Ric}|^2 - \dfrac{2}{m}R^2 + \dfrac{2(m+6)}{m}\lambda R - \dfrac{24}{m}\lambda^2\right] dV_g\right)e^{-Kt}dt\nonumber\\
			&\leq e^{K b} \int_{a}^{b} \left(\int_{D(t)} \left[- \dfrac{(m-1)}{2m}R^2 - \dfrac{2}{m}R^2 + \dfrac{2(m+6)}{m}\lambda R - \dfrac{24}{m}\lambda^2\right] dV_g\right)e^{-Kt}dt\nonumber\\
			&= e^{K b} \int_{a}^{b} \left(\int_{D(t)} \left[- \dfrac{(m+3)}{2m}R^2 + \dfrac{2(m+6)}{m}\lambda R - \dfrac{24}{m}\lambda^2\right] dV_g\right)e^{-Kt}dt\nonumber\\
			&= e^{K b} \int_{a}^{b} \left(\int_{D(t)} \left[ \dfrac{2m}{m+3}\lambda^2 - \dfrac{m + 3}{2m}\left( R - \dfrac{2(m + 6)}{m + 3}\lambda\right)^2 \right] dV_g\right)e^{-Kt}dt\nonumber\\
			&\leq e^{K b} \int_{a}^{b} \left(\int_{D(t)} \dfrac{2m}{m+3}\lambda^2 dV_g\right)e^{-Kt}dt\nonumber\\
			&\leq e^{K b}\dfrac{2m}{m+3}\lambda^2 \vol(M) \int_{a}^{b}e^{-Kt}dt = e^{K b}\dfrac{2m}{m+3}\lambda^2 \vol(M) \left( \dfrac{e^{-Ka} - e^{-Kb}}{K} \right)\nonumber\\
			&= \dfrac{2m^2}{(m+3)(m+2)}\lambda^2 \vol(M) \left( e^{(b - a)K} - 1\right).
		\end{align}
		where we used $|\ric|^2 = |\rict|^2 + \frac{R^2}{4} \ge \frac{R^2}{4}$. Noticing that $\phi'(b) = \int_M \langle \nabla R, \nabla f \rangle dV_g$, we have:
		\begin{align*}
			\int_M \langle \nabla R, \nabla f \rangle dV_g \leq \dfrac{2m^2}{(m+3)(m+2)}\lambda^2 \vol(M) \left( e^{(b - a)K} - 1\right).
		\end{align*}	
		Using Lemma~\ref{Lemma2} along with the previous inequality:
		\begin{align}\label{Euler.Theo2}
			8 \pi^2\chi(M) &= \int_M |W|^2 dV_g + \dfrac{(m- 10)\lambda}{6m(m - 1)} \int_M|\nabla f|^2 dV_g + \dfrac{m + 2}{6m(m - 1)}\int_M R|\nabla f|^2 dV_g\nonumber\\
			&\quad - \dfrac{m + 2}{12(m - 1)}\int_M \langle \nabla R, \nabla f \rangle dV_g + \dfrac{2}{3}\lambda^2 \vol(M)\nonumber\\
			&\geq \int_M |W|^2 dV_g + \dfrac{(m- 10)\lambda}{6m(m - 1)} \int_M|\nabla f|^2 dV_g + \dfrac{m + 2}{6m(m - 1)}\int_M R|\nabla f|^2 dV_g \nonumber\\
			&\quad + \dfrac{2}{3}\lambda^2 \vol(M) + \dfrac{m^2\lambda^2}{6(m - 1)(m + 3)} \vol(M) \left(1 - e^{\frac{(b - a)(m + 2)}{m}}\right)\nonumber\\
			&= \int_M |W|^2 dV_g + \dfrac{(m- 10)\lambda}{6m(m - 1)} \int_M|\nabla f|^2 dV_g + \dfrac{m + 2}{6m(m - 1)}\int_M R|\nabla f|^2 dV_g \nonumber\\
			&\quad + \dfrac{m^2\lambda^2}{6(m - 1)(m + 3)} \vol(M)\left(\dfrac{5m^2 + 8m - 12}{m^2} - e^{\frac{(b - a)(m + 2)}{m}}\right). 		
		\end{align}
		Similarly to the argument used in Proposition~\ref{proposition1.1}, we apply the bound $R \geq \frac{12}{m+3}\lambda$ to estimate the gradient terms in \eqref{Euler.Theo2}. This yields:
		\begin{align}\label{Inequality.Theo.2}
			8 \pi^2\chi(M) &\geq \int_M |W|^2 dV_g + \left(\dfrac{m- 10}{6m(m - 1)} + \dfrac{12(m + 2)}{6m(m+3)(m-1)}\right)\lambda \int_M|\nabla f|^2 dV_g \nonumber\\
			&\quad + \dfrac{m^2\lambda^2}{6(m - 1)(m + 3)} \vol(M)\left(\dfrac{5m^2 + 8m - 12}{m^2} - e^{\frac{(b - a)(m + 2)}{m}}\right)\nonumber\\
			&=\int_M |W|^2 dV_g + \dfrac{m + 6}{6m(m+3)}\lambda \int_M|\nabla f|^2 dV_g \nonumber\\
			&\quad + \dfrac{m^2\lambda^2}{6(m - 1)(m + 3)} \vol(M)\left(\dfrac{5m^2 + 8m - 12}{m^2} - e^{\frac{(b - a)(m + 2)}{m}}\right).
		\end{align}
		Since $\dfrac{m+6}{6m(m+3)}\lambda$ is a positive constant for $m>1$, the second term in \eqref{Inequality.Theo.2} is strictly non-negative. Dropping this term, we immediately obtain the desired inequality:
		\begin{align}\label{Inequality.Theo.3}
			8 \pi^2\chi(M) &\geq \int_M |W|^2 dV_g + \dfrac{m^2\lambda^2}{6(m - 1)(m + 3)} \vol(M)\left(5 + \dfrac{8}{m} - \dfrac{12}{m^2} - e^{\frac{(b - a)(m + 2)}{m}}\right).
		\end{align}
		Furthermore, equality in \eqref{Inequality.Theo.3} holds if and only if $f$ is constant. If $f$ is constant, equality follows trivially from \eqref{Euler.Grad}. Conversely, if $f$ is not constant, the strict sub-level set $D(t)$ satisfies $\vol(D(t)) < \vol(M)$ for all $t \in (a, b)$. Since the integration interval has positive length, bounding the volume of $D(t)$ by $\vol(M)$ in \eqref{Estimative.Oscf} yields a strict inequality, forcing \eqref{Inequality.Theo.3} to be strict.
	\end{proof}
	
	\begin{proof}[Proof of Corollary~\ref{CorTheorem2}]
		Using the formula
		\begin{align*}
			12\pi^2\tau(M) = \int_M (|W^+|^2 - |W^-|^2) dV_g,
		\end{align*}
		together with \eqref{Desig.Teo.2} and since $|W^{\pm}|^2$ is non-negative, we deduce
		\begin{align*}
			4\pi^2(2\chi(M) \pm 3\tau(M)) \geq \dfrac{m^2\lambda^2}{6(m - 1)(m + 3)}\vol(M) \left(5 + \dfrac{8}{m} - \dfrac{12}{m^2} - e^{\frac{f_{osc}(m + 2)}{m}}\right).
		\end{align*}
		Thus, the Hitchin--Thorpe inequality is satisfied if $0 \leq \left(5 + \frac{8}{m} - \frac{12}{m^2}\right) - e^{\frac{f_{osc}(m + 2)}{m}},$ i.e
		\begin{align*}
			f_{\max} - f_{\min} &\leq \dfrac{m}{m+2}\log\left(5 + \dfrac{8}{m} - \dfrac{12}{m^2}\right).
		\end{align*}
	\end{proof}	
	
	Before proceeding with the proof of Theorem~\ref{Theorem3}, we must ensure that the constants appearing in the denominators of our estimates are strictly positive. 
	
	\begin{lemma} \label{Lemma:Constants}
		Let $(M^n,g,f,m)$ be a compact non-trivial $m$-quasi-Einstein manifold. If $c$ and $C$ denote the global minimum and maximum of the Ricci curvature, respectively, then $c < \lambda < C$. In particular, the constants $K = \frac{\lambda - c}{m}$ and $H = \frac{C - \lambda}{m}$ are strictly positive.
	\end{lemma}
	
	\begin{proof}
		Taking the trace of the fundamental $m$-quasi-Einstein equation, we have:
		$$ R + \Delta f - \frac{1}{m}|\nabla f|^2 = n\lambda. $$
		At the global minimum point $p$ of $f$, we have $\nabla f(p) = 0$ and $\Delta f(p) \ge 0$. Thus, the trace equation yields $R(p) \le n\lambda$. Recalling that the bounds on the Ricci curvature imply $nc \le R \le nC$ throughout the manifold, we obtain $nc \le R(p) \le n\lambda$, which gives $c \le \lambda$. Similarly, at the global maximum point $q$, we find $R(q) \ge n\lambda$, which implies $\lambda \le C$. 
		
		We now show that these inequalities must be strict. Consider the auxiliary function $u = e^{-f/m}$. A direct computation yields:
		$$ \Delta u = -\frac{1}{m} u \left( \Delta f - \frac{1}{m}|\nabla f|^2 \right). $$
		Substituting the trace equation into the above identity, we obtain a linear equation for $u$:
		$$ \Delta u = \frac{1}{m}(R - n\lambda)u. $$
		
		Suppose by contradiction that $c = \lambda$. Since $R \ge nc$ everywhere, this implies $R \ge n\lambda$. Because $u > 0$, it follows that $\Delta u \ge 0$. Thus, $u$ is a subharmonic function on a compact manifold without boundary, which implies $u$ must be constant. Consequently, $f$ is constant, contradicting the non-triviality assumption. Therefore, $c < \lambda$.
		
		Similarly, assume $\lambda = C$. Since $R \le nC$ everywhere, we have $R \le n\lambda$, which forces $\Delta u \le 0$. This means $u$ is superharmonic on a compact manifold, and thus must be constant, again yielding a contradiction. 
		
		We conclude that $c < \lambda < C$.
	\end{proof}
	
	With this strict positivity established, we can now prove the diameter estimates.
	
	\begin{proof}[Proof of Theorem~\ref{Theorem3}]
		Let $p, q \in M$ be the points where $f$ attains its global minimum and global maximum, respectively, so that $f(p) = f_{\min}$ and $f(q) = f_{\max}$. At both extremal points, the gradient vanishes: $\nabla f(p) = \nabla f(q) = 0$.
		
		Evaluating the $m$-quasi-Einstein equation along a minimizing geodesic $\gamma(s)$ parameterized by arc length, we have:
		\begin{align*}
			f''(s) - \frac{1}{m}(f'(s))^2 &= \lambda - \operatorname{Ric}(\gamma', \gamma').
		\end{align*}
		We introduce the auxiliary function $u(s) = e^{-\frac{1}{m} f(\gamma(s))}$. Differentiating twice gives
		\begin{align}\label{EDOu}
			u''(s) &= -\frac{1}{m} \left( f''(s) - \frac{1}{m} (f'(s))^2 \right) u(s).
		\end{align}
		Substituting the previous equation we obtain:
		\begin{align*}
			u''(s) + \frac{\lambda - \operatorname{Ric}(\gamma',\gamma')}{m} u(s) &= 0.
		\end{align*}
		Since $u(s)>0$ and $c \le \operatorname{Ric}(v,v) \le C$, the previous equation provides differential inequalities for $u$ along $\gamma$.
		
		\bigskip
		
		\noindent \textbf{Case (1).} Let $\gamma:[0,L]\to M$ be a minimizing geodesic from $p$ to $q$, so that $L=d(p,q)\le d$, and define $K=\frac{\lambda-c}{m}$. Since $\operatorname{Ric}(\gamma',\gamma') \ge c$, from Equation~\eqref{EDOu} we obtain:
		\begin{align*}
			u''(s)+K\,u(s) &= \left(K-\frac{\lambda-\operatorname{Ric}(\gamma',\gamma')}{m}\right)u(s) \ge 0.
		\end{align*}
		If $L \ge \frac{\pi}{2\sqrt{K}}$, then the desired inequality (1) is immediate since
		\begin{align*}
			\arccos(e^{-f_{osc}/m}) < \frac{\pi}{2}
		\end{align*}
		for any non-constant potential function, because $0 < e^{-f_{osc}/m} < 1$. Thus we may assume $L < \frac{\pi}{2\sqrt{K}}$.
		
		Let $v$ be the solution of the initial value problem in $\left[0,\frac{\pi}{2\sqrt{K}}\right)$ given by
		\begin{align*}
			v''(s)+K\,v(s)=0, \qquad 
			v(0)=u(0), \qquad v'(0)=0.
		\end{align*}
		Then
		\begin{align*}
			v(s)=u(0)\cos(\sqrt{K}s).
		\end{align*} 
		Define the quotient $\phi(s)=\frac{u(s)}{v(s)}$. Since $v(s)>0$ for $s\in[0,L]$, this function is well-defined. A direct computation using the equations satisfied by $u$ and $v$ shows that:
		\begin{align*}
			\phi''(s) + 2\frac{v'(s)}{v(s)}\phi'(s)
			=
			\frac{u''+Ku}{v}
			\ge 0.
		\end{align*}
		Since $v(s) > 0$ for all $s \in [0, L]$, we can multiply both sides of the inequality above by $v(s)^2 > 0$ without changing the direction of the inequality, yielding
		\begin{align*}
			v(s)^2 \phi''(s) + 2v(s)v'(s)\phi'(s) \ge 0.
		\end{align*}
		By the product rule, we reorganize the left-hand side as the exact derivative of the function $v(s)^2 \phi'(s)$. Therefore, the expression becomes
		\begin{align*}
			\left( v(s)^2 \phi'(s) \right)' \ge 0.
		\end{align*}
		Integrating both sides from $0$ to $t$, with $t \in (0, L]$, we obtain
		\begin{align*}
			v(t)^2 \phi'(t) - v(0)^2 \phi'(0) &\ge 0.
		\end{align*}
		Since $\nabla f(p) = 0$, we have $u'(0) = 0$. Together with $v'(0) = 0$, this implies $\phi'(0) = 0$. Thus, the inequality reduces to
		\begin{align*}
			v(t)^2 \phi'(t) \ge 0.
		\end{align*}
		Since $v(t)^2 > 0$, we necessarily conclude that $\phi'(t) \ge 0, \text{ for all } t \in [0, L].$
		
		Since the derivative is non-negative, it follows that $\phi$ is a non-decreasing function on the considered interval. Furthermore, since $\phi(0) = \frac{u(0)}{v(0)} = 1$, we have $\phi(s) \ge 1$, which yields
		\begin{align*}
			u(s) \ge v(s).
		\end{align*}
		Evaluating at $s=L$ (point $q$) gives $e^{-f_{\max}/m} \ge e^{-f_{\min}/m}\cos(\sqrt{K}L)$, or equivalently,
		\begin{align*}
			e^{-f_{osc}/m} \ge \cos(\sqrt{K}L).
		\end{align*}
		Since $L < \frac{\pi}{2\sqrt{K}}$, the cosine term is positive. Applying the strictly decreasing $\arccos$ function we obtain $\arccos(e^{-f_{osc}/m}) \le \sqrt{K}L$, i.e.,
		\begin{align*}
			L \ge \sqrt{\frac{m}{\lambda-c}} \arccos\left(e^{-\frac{f_{osc}}{m}}\right).
		\end{align*}
		Since $d \ge L = d(p,q)$, the first estimate follows.
		
		\bigskip
		
		\noindent \textbf{Case (2).} Now let $H = \frac{C-\lambda}{m}$ and parameterize the minimizing geodesic so that $\gamma(0)=q$. Using $\operatorname{Ric}(\gamma', \gamma') \le C$, we obtain $u''(s) \le H u(s)$. The initial conditions are $u(0) = e^{-f_{\max}/m}$ and $u'(0) = 0$. Let $v$ be the solution of the initial value problem
		\begin{align*}
			v''(s) - H\,v(s) = 0, \qquad v(0) = u(0), \qquad v'(0) = 0.
		\end{align*}
		Then
		\begin{align*}
			v(s) = u(0)\cosh\left(\sqrt{H}s\right).
		\end{align*} 
		Defining the quotient $\phi(s) = \frac{u(s)}{v(s)}$, a direct computation using the equations satisfied by $u$ and $v$ yields:
		\begin{align*}
			\phi''(s) + 2\frac{v'(s)}{v(s)}\phi'(s) = \frac{u''(s) - H u(s)}{v(s)} \le 0,
		\end{align*}
		where we used the fact that $v(s) = u(0)\cosh(\sqrt{H}s) > 0$ for all $s$.
		Since $\cosh(x) > 0$ for all $x$, a similar argument (yielding $\phi'(t) \le 0$) implies $u(s) \le v(s)$ for all $s$. Evaluating at $s=L$ (point $p$), we obtain $e^{-f_{\min}/m} \le e^{-f_{\max}/m} \cosh\left(\sqrt{H}L\right)$, rearranging the terms
		\begin{align*}
			e^{\frac{f_{osc}}{m}} \le \cosh\left(\sqrt{H}L\right).
		\end{align*}
		Applying the strictly increasing $\operatorname{arccosh}$ function yields:
		\begin{align*}
			\operatorname{arccosh}\left(e^{\frac{f_{osc}}{m}}\right) \le \sqrt{H}L = \sqrt{\frac{C-\lambda}{m}}L.
		\end{align*}
		Since $d \ge L$, the second estimate follows.
	\end{proof}

	\begin{proof}[Proof of Proposition~\ref{Prop:MixedEstimate}]
		Let $\gamma : [0, L] \to M$ be a minimizing geodesic from $p$ to $q$, where $f$ attains its global minimum and maximum, respectively. Consider the midpoint $x = \gamma(L/2)$. We split $\gamma$ into two geodesic segments of length $L/2$: one from $p$ to $x$, and another from $q$ to $x$.
		
		Since $L \le d$ and the diameter bound is $d < \pi\sqrt{\frac{m}{\lambda - c}}$, the length of the segment from $p$ to $x$ strictly satisfies:
		\begin{equation*}
			\frac{L}{2} \le \frac{d}{2} < \frac{\pi}{2}\sqrt{\frac{m}{\lambda - c}}.
		\end{equation*}
		This ensures that the comparison argument from Theorem~\ref{Theorem3} is valid along the first segment $\gamma|_{[0, L/2]}$. Applying this estimate starting from $p$, we get:
		\begin{equation*}
			e^{-\frac{f(x) - f(p)}{m}} \ge \cos\left( \sqrt{\frac{\lambda - c}{m}} \frac{L}{2} \right).
		\end{equation*}
		Recalling that $f(p) = f_{\min}$ and taking the inverse of both sides, the inequality is reversed and the cosine becomes a secant, yielding the upper bound for the first half:
		\begin{equation}\label{eq:half1}
			e^{\frac{f(x) - f_{\min}}{m}} \le \sec\left( \sqrt{\frac{\lambda - c}{m}} \frac{L}{2} \right).
		\end{equation}
		Analogously, we apply the corresponding hyperbolic comparison estimate from Theorem \ref{Theorem3} to the second segment, parameterizing the geodesic from $q$ to $x$. Since $f(q) = f_{\max}$ and the gradient vanishes at $q$, this directly bounds the variation of $f$ on the second half by:
		\begin{equation}\label{eq:half2}
			e^{\frac{f_{\max} - f(x)}{m}} \le \cosh\left( \sqrt{\frac{C - \lambda}{m}} \frac{L}{2} \right).
		\end{equation}
		Multiplying inequalities \eqref{eq:half1} and \eqref{eq:half2}, the terms involving $f(x)$ cancel out in the exponent, yielding:
		\begin{equation*}
			e^{\frac{f_{\max} - f_{\min}}{m}} \le \sec\left( \sqrt{\frac{\lambda - c}{m}} \frac{L}{2} \right) \cosh\left( \sqrt{\frac{C - \lambda}{m}} \frac{L}{2} \right).
		\end{equation*}
		Noting that $f_{\max} - f_{\min} = f_{osc}$ and that $L \le d$, the strictly increasing monotonicity of both $\sec(t)$ and $\cosh(t)$ for non-negative arguments within the considered interval yields the desired estimate:
		\begin{equation*}
			e^{\frac{f_{osc}}{m}} \le \sec\left( \sqrt{\frac{\lambda - c}{m}} \frac{d}{2} \right) \cosh\left( \sqrt{\frac{C - \lambda}{m}} \frac{d}{2} \right).
		\end{equation*}
		This completes the proof.
	\end{proof}

	\begin{proof}[Proof of Corollary~\ref{CorTheorem3}]
		Consider a $4$-dimensional $m$-quasi-Einstein manifold $(M^4, g, f, \lambda)$ with $m > 1$. By assumption, the Hitchin--Thorpe inequality holds whenever $e^{f_{osc}/m} \le D_m.$ We shall show that each diameter constraint implies this condition.
		
		\textbf{Case (1):} Suppose $d \le \sqrt{\frac{m}{\lambda - c}} \arccos\left( \frac{1}{D_m} \right)$. From Theorem~\ref{Theorem3}, we have the diameter estimate:
		\begin{align*}
			d \ge \sqrt{\frac{m}{\lambda - c}} \arccos\left( e^{-\frac{f_{osc}}{m}} \right).
		\end{align*}
		Combining these inequalities, we obtain:
		\begin{align*}
			\sqrt{\frac{m}{\lambda - c}} \arccos\left( e^{-\frac{f_{osc}}{m}} \right) \le d \le \sqrt{\frac{m}{\lambda - c}} \arccos\left( \frac{1}{D_m} \right).
		\end{align*}
		Dividing by the constant factor, we have $\arccos\left( e^{-f_{osc}/m} \right) \le \arccos(1/D_m)$. Since the inverse cosine function is strictly decreasing on $[0, 1]$, applying the cosine function to both sides reverses the inequality, yielding $e^{-f_{osc}/m} \ge 1/D_m$. Inverting both sides, we conclude
		\begin{align*}
			e^{\frac{f_{osc}}{m}} \le D_m.
		\end{align*}
		Thus, the Hitchin--Thorpe inequality holds.
		
		\textbf{Case (2):} Suppose $d \le \sqrt{\frac{m}{C - \lambda}} \operatorname{arccosh}(D_m)$. Using the second estimate from Theorem~\ref{Theorem3}:
		\begin{align*}
			\sqrt{\frac{m}{C - \lambda}} \operatorname{arccosh}\left( e^{\frac{f_{osc}}{m}} \right) \le d \le \sqrt{\frac{m}{C - \lambda}} \operatorname{arccosh}(D_m).
		\end{align*}
		It follows that $\operatorname{arccosh}\left( e^{f_{osc}/m} \right) \le \operatorname{arccosh}(D_m)$. Since the inverse hyperbolic cosine is strictly increasing for arguments $\ge 1$, we maintain the inequality direction:
		\begin{align*}
			e^{\frac{f_{osc}}{m}} \le D_m,
		\end{align*}
		which ensures the Hitchin--Thorpe inequality.
		
		\textbf{Case (3):} Let $H(x) = \cosh\left( \sqrt{\frac{C - \lambda}{m}} x \right) \sec\left( \sqrt{\frac{\lambda - c}{m}} x \right)$. In the interval $\left(0, \frac{\pi}{2}\sqrt{\frac{m}{\lambda - c}}\right)$, $H(x)$ is the product of two positive, strictly increasing functions, and thus $H(x)$ is strictly increasing. 
		From Proposition~\ref{Prop:MixedEstimate}, we have $e^{f_{osc}/m} \le H(d/2).$ By hypothesis, $d \le 2x_0$, which implies $d/2 \le x_0$. Due to the monotonicity of $H$, we have:
		\begin{align*}
			e^{\frac{f_{osc}}{m}} \le H(d/2) \le H(x_0).
		\end{align*}
		Since $x_0$ is the unique solution to $H(x_0) = D_m$, we conclude $e^{f_{osc}/m} \le D_m$, and the result follows.
	\end{proof}
	
	\begin{proof}[Proof of Theorem~\ref{Theorem.Est.Vol}]
		As shown in \cite{Case2011}, the scalar curvature $R$ is strictly positive. For a compact oriented $4$-manifold with positive scalar curvature, a result by Gursky \cite{Gursky1994} (see also \cite{Seshadri2006}) establishes that:
		\begin{align}\label{Eq.1Gursky}
			8\pi^2(\chi(M) - 2) \leq \int_M |W^+|^2 dV_g \leq \int_M |W|^2 dV_g,
		\end{align}
		where the first equality holds if and only if $M$ is conformally equivalent to a sphere. Rearranging this inequality yields:
		\begin{align*}
			16\pi^2 \geq 8\pi^2\chi(M) - \int_M |W|^2 dV_g.
		\end{align*}
		Applying the estimate from Theorem~\ref{theorem2} directly to the right-hand side, we obtain:
		\begin{align*}
			16\pi^2 \geq \frac{m^2\lambda^2}{6(m - 1)(m + 3)} \left( 5 + \frac{8}{m} - \frac{12}{m^2} - e^{\frac{f_{osc}(m + 2)}{m}} \right) \vol(M).
		\end{align*}
		Multiplying both sides by $6/\lambda^2$ yields the volume estimate \eqref{Eq.1Teo4}. If equality holds, $M$ is conformally equivalent to the round sphere, making $f$ constant and $W = 0$. Consequently, $M$ is Einstein with $R = 4\lambda$, meaning it has constant sectional curvature $\lambda/3$, which identifies it as a sphere of radius $\sqrt{3/\lambda}$.
		
		To prove the Yamabe invariant estimate \eqref{Eq.2Teo4}, recall the inequality by Cheng, Ribeiro, and Zhou \cite{ChengRibeiroZhou2023}:
		\begin{align}\label{Ineq.Yamabe.ChengRibZh}
			\mathcal{Y}(M, [g])^2 \geq \int_M (R^2 - 12|\mathring{\operatorname{Ric}}|^2) dV_g.
		\end{align}
		By the generalized Chern-Gauss-Bonnet formula for $4$-manifolds, we have the identity:
		\begin{align*}
			8\pi^2\chi(M) - \int_M |W|^2 dV_g = \frac{1}{24} \int_M (R^2 - 12|\mathring{\operatorname{Ric}}|^2) dV_g.
		\end{align*}
		Substituting this identity into \eqref{Ineq.Yamabe.ChengRibZh}, we get:
		\begin{align*}
			\mathcal{Y}(M, [g])^2 \geq 24 \left( 8\pi^2\chi(M) - \int_M |W|^2 dV_g \right).
		\end{align*}
		Invoking Theorem~\ref{theorem2} once more to bound the term in parentheses, we obtain:
		\begin{align*}
			\mathcal{Y}(M, [g])^2 \geq 24 \left[ \frac{m^2\lambda^2}{6(m - 1)(m + 3)} \left( 5 + \frac{8}{m} - \frac{12}{m^2} - e^{\frac{f_{osc}(m + 2)}{m}} \right) \vol(M) \right].
		\end{align*}
		Simplifying the leading constant ($24/6 = 4$) yields exactly inequality \eqref{Eq.2Teo4}. If equality holds, it forces equality in Theorem~\ref{theorem2}, implying that $f$ is constant, and therefore $M$ is Einstein.
	\end{proof}



\begin{thebibliography}{99}
		\bibitem{bakry1985}
		Bakry, D. and \'Emery, M., \textit{Diffusions hypercontractives}, S\'eminaire de Probabilit\'es XIX 1983/84, Lecture Notes in Math. \textbf{1123}, Springer, Berlin, (1985), 177--206.
		
		\bibitem{Barros2012}
		Barros, A. and Ribeiro Jr., E., \textit{Integral formulae on quasi-Einstein manifolds and applications}, Glasg. Math. J. \textbf{54} (2012), 213--223.
		
		\bibitem{barros2014}
		Barros, A., Ribeiro Jr., E., and Silva, J., \textit{Uniqueness of quasi-Einstein metrics on 3-dimensional homogeneous manifolds}, Diff. Geom. Appl. \textbf{35} (2014), 60--73.
		
		\bibitem{Batat2015}
		Batat, W., Hall, S. J., Jizany, A., and Murphy, T., \textit{Conformally K\"{a}hler geometry and quasi-Einstein metrics}, M\"{u}nster J. Math. \textbf{8} (2015), 211--228.
		
		\bibitem{Besse}
		Besse, A. L., \textit{Einstein manifolds}, Classics in Mathematics, Springer-Verlag, Berlin, 2008.
		
		\bibitem{bohm1998}
		Böhm, C., \textit{Inhomogeneous Einstein metrics on low-dimensional spheres and other low-dimensional spaces}, Invent. Math. \textbf{134} (1998), 145--176.
		
		\bibitem{bohm1999}
		Böhm, C., \textit{Non-compact cohomogeneity one Einstein manifolds}, Bull. Soc. Math. France \textbf{127} (1999), 135--177.
		
		\bibitem{brasil2014}
		Brasil, A., Costa, E., and Ribeiro Jr., E., \textit{Hitchin--Thorpe inequality and Kaehler metrics for compact almost Ricci soliton}, Ann. Mat. Pura Appl. (4) \textbf{193} (2014), no. 6, 1851--1860.
		
		\bibitem{cao2010}
		Cao, H.-D., \textit{Recent progress on Ricci solitons}, Adv. Lect. Math. \textbf{11} (2010), 1--38.
		
		\bibitem{Case2011}
		Case, J., Shu, Y.-J., and Wei, G., \textit{Rigidity of quasi-Einstein metrics}, Differ. Geom. Appl. \textbf{29} (2011), 93--100.
		
		\bibitem{catino2012}
		Catino, G., \textit{A note on four-dimensional (anti-)self-dual quasi-Einstein manifolds}, Diff. Geom. Appl. \textbf{30} (2012), 660--664.
		
		\bibitem{catino2013}
		Catino, G., Mantegazza, C., Mazzieri, L., and Rimoldi, M., \textit{Locally conformally flat quasi-Einstein manifolds}, J. Reine Angew. Math. (Crelle's Journal) \textbf{675} (2013), 181--189.
		
		\bibitem{Cheng2020}
		Cheng, X., Ribeiro Jr., E., and Zhou, D., \textit{Volume growth estimates for Ricci solitons and quasi-Einstein manifolds}, J. Geom. Anal. \textbf{32} (2022), 62.
		
		\bibitem{ChengRibeiroZhou2023}
		Cheng, X., Ribeiro Jr., E., and Zhou, D., \textit{On Euler characteristic and Hitchin--Thorpe inequality for compact gradient Ricci solitons}, Proc. Amer. Math. Soc. \textbf{10} (2023), 33--45.
		
		\bibitem{colding2021}
		Colding, T. and Minicozzi, W., \textit{Optimal growth bounds for eigenfunctions}, arXiv preprint arXiv:2109.04998 (2021).
		
		\bibitem{Costa2025}
		Costa, J., Ribeiro Jr., E., and Santos, M., \textit{On quasi-Einstein manifolds with constant scalar curvature}, arXiv preprint arXiv:2505.18834 (2025).
		
		\bibitem{Deng2021}
		Deng, Y., \textit{Diameter estimate for a class of compact generalized quasi-Einstein manifolds}, Turk. J. Math. \textbf{45} (2021), 2331--2340.
		
		\bibitem{Deng2024}
		Deng, Y., \textit{Diameter estimate for generalized $m$-quasi-Einstein manifolds}, Ital. J. Pure Appl. Math. \textbf{52} (2024), 9--17.
		
		\bibitem{diogenes2022compact}
		Di\'ogenes, R. and Gadelha, T., \textit{Compact quasi-Einstein manifolds with boundary}, Math. Nachr. \textbf{295} (2022), 1690--1708.
		
		\bibitem{diogenes2022remarks}
		Di\'ogenes, R., Gadelha, T., and Ribeiro Jr., E., \textit{Remarks on quasi-Einstein manifolds with boundary}, Proc. Amer. Math. Soc. \textbf{150} (2022), 351--363.
		
		\bibitem{Diogenes2025}
		Di\'{o}genes, R., Gon\c{c}alves, J., and Ribeiro Jr., E., \textit{Geometric inequalities for quasi-Einstein manifolds}, Ann. Mat. Pura Appl. (2025).
		
		\bibitem{FernandezLopez2010}
		Fern\'{a}ndez-L\'{o}pez, M. and Garc\'{i}a-R\'{i}o, E., \textit{Diameter bounds and Hitchin--Thorpe inequalities for compact Ricci solitons}, Quart. J. Math. \textbf{61} (2010), 319--327.
		
		\bibitem{Gursky1994}
		Gursky, M. J., \textit{Locally conformally flat four- and six-manifolds of positive scalar curvature and positive Euler characteristic}, Indiana Univ. Math. J. \textbf{43} (1994), 747--774.
		
		\bibitem{hamilton1995}
		Hamilton, R., \textit{The formation of singularities in the Ricci flow}, Surveys in differential geometry, Vol. II (Cambridge, MA, 1993), Int. Press, Cambridge, MA, (1995), 7--136.
		
		\bibitem{He2012}
		He, C., Petersen, P., and Wylie, W., \textit{On the classification of warped product Einstein metrics}, Comm. Anal. Geom. \textbf{20} (2012), 271--311.
		
		\bibitem{he2014warped}
		He, C., Petersen, P., and Wylie, W., \textit{Warped product Einstein metrics over spaces with constant scalar curvature}, Asian J. Math. \textbf{18} (2014), 159--190.
		
		\bibitem{Hitchin1974}
		Hitchin, N., \textit{Compact Four-Dimensional Einstein Manifolds}, J. Differ. Geom. \textbf{9} (1974), 435--441.
		
		\bibitem{kim2003}
		Kim, D.-S. and Kim, Y. H., \textit{Compact Einstein warped product spaces with nonpositive scalar curvature}, Proc. Amer. Math. Soc. \textbf{131} (2003), 2573--2576.
		
		\bibitem{koiso1993}
		Koiso, N., \textit{On Rotationally Symmetric Hamilton's Equation for K\"{a}hler–Einstein Metrics}, Proc. Sympos. Pure Math. \textbf{54} (1993), 327--337.
		
		\bibitem{Limoncu2010}
		Limoncu, M., \textit{Modifications of the Ricci tensor and applications}, Arch. Math. \textbf{95} (2010), 191--199.
		
		\bibitem{LuPagePope2004}
		L\"{u}, H., Page, D. N., and Pope, C. N., \textit{New inhomogeneous Einstein metrics on sphere bundles over Einstein-K\"{a}hler manifolds}, Phys. Lett. B \textbf{593} (2004), 218--226.
		
		\bibitem{lima2013}
		Ma, L., \textit{Remarks on compact shrinking Ricci solitons of dimension four}, C. R. Math. Acad. Sci. Paris \textbf{351} (2013), 817--823.
		
		\bibitem{mastrolia2014}
		Mastrolia, P. and Rimoldi, M., \textit{Some triviality results for quasi-Einstein manifolds and Einstein warped products}, Geom. Dedic. \textbf{169} (2014), 225--237.
		
		\bibitem{qian1997}
		Qian, Z., \textit{Estimates for weighted volumes and applications}, Quart. J. Math. \textbf{48} (1997), 235--242.
		
		\bibitem{ribeiro2021}
		Ribeiro Jr., E. and Tenenblat, K., \textit{Noncompact quasi-Einstein manifolds conformal to a Euclidean space}, Math. Nachr. \textbf{294} (2021), no. 1, 132--144.
		
		\bibitem{rimoldi2011}
		Rimoldi, M., \textit{A remark on Einstein warped products}, Pacific J. Math. \textbf{252} (2011), 207--218.
		
		\bibitem{Seshadri2006}
		Seshadri, H., \textit{Weyl curvature and the Euler characteristic in dimension four}, Differ. Geom. Appl. \textbf{24} (2006), 172--177.
		
		\bibitem{Shaikh2024}
		Shaikh, A. A., Mandal, P., and Mondal, C. K., \textit{Diameter estimation of $(m,p)$-quasi Einstein manifolds}, Proc. Natl. Acad. Sci. India Sect. A Phys. Sci. \textbf{94} (2024), 513--518.
		
		\bibitem{tadano2022riccati}
		Tadano, H., \textit{$m$-Bakry–Émery Ricci curvatures, Riccati inequalities, and bounded diameters}, Differ. Geom. Appl. \textbf{80} (2022), 101832.
		
		\bibitem{tadano2018}
		Tadano, H., \textit{An upper diameter bound for compact Ricci solitons with application to the Hitchin--Thorpe inequality. II}, J. Math. Phys. \textbf{59} (2018), no. 4, 043507, 3.
		
		\bibitem{tadano2025improvement}
		Tadano, H., \textit{An improvement of the Myers theorem via $m$-Bakry–Émery Ricci curvature with $\varepsilon$-range}, Ann. Mat. Pura Appl. (4) \textbf{204} (2025), no. 6, 2757--2769.
		
		\bibitem{tadano2025myers}
		Tadano, H., \textit{Myers-type theorems via $m$-Bakry–Émery Ricci curvature of quadratic decays}, Internat. J. Math. \textbf{36} (2025), no. 14, 2450059.
		
		\bibitem{tadano2026compactness}
		Tadano, H., \textit{New Compactness Criteria via Integral Radial $m$-Bakry–Émery Ricci Curvatures}, Bull. Sci. Math. (2026), 103837.
		
		\bibitem{tadano2020compactness}
		Tadano, H., \textit{Some Cheeger--Gromov--Taylor type compactness theorems via $m$-Bakry–Émery and $m$-modified Ricci curvatures}, Nonlinear Anal. \textbf{199} (2020), 112045.
		
		\bibitem{thorpe1969}
		Thorpe, J. A., \textit{Some remarks on the Gauss-Bonnet integral}, J. Math. Mech. \textbf{18} (1969), 779--786.
		
		\bibitem{wang2011}
		Wang, L., \textit{On noncompact $m$-quasi-Einstein metrics}, Pacific J. Math. \textbf{254} (2011), 449--464.
		
		\bibitem{Wang2013}
		Wang, L. F., \textit{Diameter estimate for compact quasi-Einstein metrics}, Math. Z. \textbf{273} (2013), 801--809.
		
		\bibitem{wang2014}
		Wang, L. F., \textit{Potential function estimates for quasi-Einstein metrics}, J. Funct. Anal. \textbf{267} (2014), 1986--2004.
		
		\bibitem{WangZhu2004}
		Wang, X.-J. and Zhu, X., \textit{K\"{a}hler–Ricci solitons on toric manifolds with positive first Chern class}, Adv. Math. \textbf{188} (2004), 87--103.
		
		\bibitem{WeiWylie2007}
		Wei, G. and Wylie, W., \textit{Comparison Geometry for the Smooth Metric Measure Spaces}, Proc. 4th Int. Congr. Chin. Math. (ICCM) \textbf{II} (2007), 1--4.
		
		\bibitem{Wu2018}
		Wu, J.-Y., \textit{Myers' type theorem with the Bakry-\'{E}mery Ricci tensor}, Ann. Glob. Anal. Geom. \textbf{54} (2018), 541--549.
		
		\bibitem{zhang2012}
		Zhang, Y., \textit{A note on the Hitchin--Thorpe inequality and Ricci flow on 4-manifolds}, Proc. Amer. Math. Soc. \textbf{140} (2012), no. 5, 1777--1783.
	\end{thebibliography}
\end{document}